\documentclass[a4paper,11pt]{amsart}

\usepackage{ifpdf}
\usepackage{amsmath, amsthm, amsfonts, amssymb, amscd}
\usepackage[active]{srcltx}
\usepackage{color}
\usepackage{cancel}
\usepackage[ansinew]{inputenc}
\usepackage{graphicx}

\newtheorem{theorem}{Theorem}[section] 
\newtheorem{corollary}[theorem]{Corollary}
\newtheorem{lemma}[theorem]{Lemma}

\newtheorem{proposition}[theorem]{Proposition}
\newtheorem*{proposition*}{Proposition}

\newtheorem*{question*}{Question}
\newtheorem*{theorem*}{Theorem}
\newtheorem*{claim*}{Claim}

\newtheorem*{corollary*}{Corollary}

\theoremstyle{definition}

\theoremstyle{remark}

\newtheorem*{remark*}{Remark}

\newcommand{\R}{\mathbb{R}}\newcommand{\N}{\mathbb{N}}
\newcommand{\Z}{\mathbb{Z}}\newcommand{\Q}{\mathbb{Q}}

\newcommand{\D}{\mathbb{D}}
\renewcommand{\S}{\mathbb{S}}

  \def\cG{\mathcal{G}}  
  \def\cH{\mathcal{H}}

\def\cE{\mathcal{E}}

\DeclareMathOperator{\Diff}{Diff}
\DeclareMathOperator{\Ham}{Ham}

\newcommand{\ka}{\ensuremath{\kappa}}

\begin{document}

\author[P. Le Calvez]{Patrice Le Calvez}
\address{Institut de Math\'ematiques de Jussieu-Paris Rive Gauche, IMJ-PRG, Sorbonne Universit\'e, Universit\'e Paris-Diderot, CNRS, F-75005, Paris, France \enskip \& \enskip Institut Universitaire de France}
\curraddr{}
\email{patrice.le-calvez@imj-prg.fr}

\author[M. Sambarino]{Mart\'{\i}n Sambarino$^*$}
\address{CMAT, Facultad de Ciencias, Universidad de la Rep\'ublica, Uruguay.}
\curraddr{Igua 4225 esq. Mataojo. Montevideo, Uruguay.}
\email{samba@cmat.edu.uy}
\thanks{$^*$ Partially supported by CSIC Group 618, Universidad de la Rep\'ublica}
\title[Non contractible periodic orbits]{Non contractible periodic orbits for generic Hamiltonian diffeomorphisms of surfaces}

\begin{abstract}
Let $S$ be a closed surface of genus $g\geq 1$, furnished with an area form $\omega$. We show that there exists an open and dense set ${\mathcal O_r}$ of the space of Hamiltonian diffeomorphisms of class $C^r$, $1\leq r\leq\infty$, endowed with the $C^r$-topology, such that every $f\in \mathcal O_r$ possesses infinitely many non contractible periodic orbits. We obtain a positive answer to a question asked by Viktor Ginzburg and Ba\c{s}ak G\"{u}rel. The proof is a consequence of recent previous works of the authors \cite{LecSa}. \end{abstract}

\maketitle



\tableofcontents
\section{Introduction}

\subsection{The main theorem}\label{ss.maintheorem}

In this article, $S$ will denote a smooth compact boundaryless orientable surface of genus $g\geq 1$, furnished with a smooth area form $\omega$.  We will assume for conveniency that $\int_S\omega=1$. For $1\leq r\leq \infty$, denote $\Diff^r_{\omega}(S)$ the set of $C^r$-diffeomorphisms of $S$ preserving $\omega$ endowed with the $C^r$-topology, and $\Diff^r_{\omega, *}(S)$ the connected component of $\Diff^r_{\omega}(S)$ that contains the identity map. It is well known that $\Diff^r_{\omega, *}(S)$ is path connected: for every $f\in \Diff^r_{\omega, *}(S)$, there exists a continuous path $I=(f_t)_{t\in[0,1]}$ in $\Diff^r_{\omega}(S)$ that joins the identity to $f$. Such a path is called an {\it identity isotopy} of $f$. 

First, we will recall the definition of the {\it rotation vector} of a finite Borel measure $\mu$ of $S$ invariant by $f$, when $f\in\Diff^r_{\omega, *}(S)$ and  $I=(f_t)_{t\in[0,1]}$ is an identity isotopy of $f$ (see \cite{Mats}, \cite {Po} or \cite{Sc}). Let us define the {\it trajectory} of a point $z\in S$ to be the path $I(z): t\mapsto f_t(z)$. If $\alpha$ is a closed $1$-form  and if $\gamma:[0,1]\to S$ is a $C^1$-path homotopic to $I(z)$ relative to the endpoints, then  $\int_{\gamma} \alpha$ does not depend on $\gamma$ and will be denoted $\int_{I(z)} \alpha$. This quantity depends only on the homotopy class of $I$ (relative to the endpoints). One defines a real valued morphism $\alpha\to \int_S \left(\int_{I(z)} \alpha\right) \, d\mu(z)$ on the space of closed $1$-forms that vanishes on the space of exact $1$-forms (because $\mu$ is invariant by $f$). Passing to the quotient space, one gets in that way a natural linear form on the cohomology group $H^1(S, \R)$. One obtains by duality a homology class $\mathrm{rot}_{I}(\mu)\in H_1(S,\R)$, which is the rotation vector of $\mu$. In the case where $g\geq 2$, it is known that $\Diff^r_{\omega, *}(S)$ is simply connected. Consequently $\mathrm{rot}_{I}(\mu)$ does not depend on $I$ and can be written $\mathrm{rot}_{f}(\mu)$.  In the case where $g=1$, which means that $S$ is the $2$-torus, $\mathrm{rot}_{I}(\mu)$ depends on $I$ but  $$\mathrm{rot}_{f}(\mu)= \mathrm{rot}_{I}(\mu)+H_1(S,\Z)\in H_1(S,\R)/H_1(S,\Z)$$ does not. The set ${\mathcal M}(f)$ of Borel probability measures of $S$ invariant by $f$ is non empty, convex and compact for the weak$^*$-topology. So, the set $\{\mathrm{rot}_I(\mu)\,\vert \enskip \mu\in {\mathcal M}(f)\}$ is a non empty compact convex subset of $H_1(S,\R)$ called the {\it rotation set}.

 An interesting case is the case where $\mu=\mu_O$ is a probability measure supported on a periodic orbit $O$. We will speak about rotation vector of $O$ and will use the notations $\mathrm{rot}_{I}(O)$ and $\mathrm{rot}_{f}(O)$ instead of $\mathrm{rot}_{I}(\mu_O)$ and $\mathrm{rot}_{f}(\mu_O)$. Another interesting case is the case where $\mu=\mu_{\omega}$ is the Borel measure naturally defined by $\omega$. Say that $I$ is {\it Hamiltonian} if $\mathrm{rot}_{I}(\mu_{\omega})=0$. It is well known that $I$ is Hamiltonian if and only if $I$ is homotopic to an identity isotopy $I'=(f'_t)_{t\in[0,1]}$ satisfying the following:
 There exists a family $(H_t)_{t\in[0,1]}$ of functions of class $C^{r+1}$ and a time dependent vector field $(X_t)_{t\in[0,1]}$ of class $C^r$ such that
\begin{itemize}
\item for every $z\in S$, every $v\in T_zS$ and every $t\in [0,1]$,  it holds that  $dH_t(z).v=\omega(v, X_t(z))$,
\item for every $z\in S$ it holds that ${d\over dt} f'_t(z)= X_t(f'_t(z))$.
\end{itemize}
Say that $f\in\Diff^r_{\omega, *}(S)$ is Hamiltonian if $\mathrm{rot}_{f}(\mu_{\omega})=0$. In the case where $g\geq 2$, it means that every identity isotopy of $f$ is Hamiltonian; in the case where $g=1$ it means that there exists an identity isotopy of $f$ that is Hamiltonian. In both cases $g\geq 2$ and $g=1$ one proves easily that the map $f\mapsto \mathrm{rot}_{f}(\mu_{\omega})$, defined on $\Diff^r_{\omega, *}(S)$, is continuous. Consequently, the space of Hamiltonian diffeomorphisms endowed with the $C^r$-topology, denoted  $\Ham^r_{\omega}(S)$, is a Baire space. 

Let us state the main result of the article

\begin{theorem}\label{t.main}
 For $1\leq r\leq \infty$, there exists an open and dense set ${\mathcal O}_r\subset\Ham^r_{\omega}(S)$ such that if $f\in{\mathcal O}_r$, there exist infinitely many periodic orbits $O$ such that $\mathrm{rot}_{f}(O)\not=0$.
\end{theorem}

\begin{remark*} The result above was already proved by Tal and the first author in the case of the torus (see \cite{LecTal}). The present proof is valid for surfaces of higher genus, but works also for the torus. In this last case, it is completely different from the original proof, that used  the notion of transverse foliation.
\end{remark*}

\subsection{A more precise statement}\label{ss.maintheorem}

Let $f$ be a diffeomorphism in $\Diff_\omega^r(S)$ and let $z$ be a periodic point of period $q$\footnote{In the article ``period'' means ``smallest period''.}. We say that $z$ is {\it elliptic} if $Df^q(z)$ has two non real complex eigenvalues of modulus one, and that $z$ is {\it hyperbolic} if $Df^q(z)$ has two real eigenvalues of modulus different from one. In the last case, one can define the {\it stable manifold} $W^s(z)$ and the {\it unstable manifold} $W^u(z)$, setting
$$\begin{aligned}W^s(z)&=\{z'\in S\, \vert\, \lim_{n\to+\infty} f^{nq}(z')=z\}, \\ W^u(z)&=\{z'\in S\, \vert\, \lim_{n\to+\infty} f^{-nq}(z')=z\}.\end{aligned}$$
Both sets $W^s(z)$ and $W^u(z)$ are the images of $C^r$ immersions of $\R$. A {\it stable branch} $\Lambda_s$ is a connected component of $W^s(z)\setminus\{z\}$ and an {\it unstable branch} $\Lambda_u$ a connected component of $W^u(z)\setminus\{z\}.$ We refer as a branch any of the stable or unstable branches. If $z$ is a hyperbolic periodic point of period $q$, there are two possibilities. In the first case, the eigenvalues of $Df^q(z)$ are positive and the four branches are invariant by $f^q$: we will say that $z$ is a {\it hyperbolic periodic point with no reflection}. In the second case, the eigenvalues of $Df^q(z)$ are negative and the stable branches are exchanged by $f^q$ as the unstable branches: we will say that $z$ is a {\it hyperbolic periodic point with reflection}. A hyperbolic periodic point has a {\it transverse homoclinic intersection} if there is a stable branch that meets an unstable branch in a transversal way. In that case there are infinitely many hyperbolic periodic points, due to the presence of a horseshoe (see \cite{Sm}).

 We denote $\cG^r_{\omega}(S)$ the set of diffeomorphisms $f\in\mathrm{Diff}^r_\omega(S)$ satisfying the following conditions.

\begin{itemize}

\item[(G1):]\label{n.genericity} Every periodic point is either elliptic or hyperbolic. Moreover, if $z$ is an elliptic periodic point of period $q$, then the eigenvalues of $Df^q(z)$ are not roots of unity.

 \item[(G2):]\label{h.genericity}  Stable and unstable branches of hyperbolic points that intersect must also intersect transversally.
  \item[(G3):]\label{e.genericity} If $U$ is a neighborhood of an elliptic periodic point $z$, then there
is a topological closed disk $D$ containing $z$, contained in $U$, and bordered by finitely
many pieces of stable and unstable manifolds of some hyperbolic periodic
point $z'$.
\end{itemize}
Moreover we define $$\cH^r_{\omega}(S)=\cG^r_{\omega}(S)\cap\Ham^r_{\omega}(S).$$
Note that if $f$ belongs to $\cG^r_{\omega}(S)$ it is also the case for every $f^k$, $k\not =0$. We have a similar result for $\cH^r_{\omega}(S)$.

Robinson \cite{R} proved that, for any $r \ge 1$, properties (G1) and (G2) are $C^r$-generic  (it is easy to see that the no unity root condition is generic among elliptic periodic points), and (G3) is $C^r$-generic due to Zehnder \cite{Z}. Thus $\cG^r_{\omega}(S)$ is residual in $\Diff^r_\omega(S).$ The perturbations that permit to approximate an element $f\in \Diff^r_\omega(S)$ by an element of  $\cG^r_{\omega}(S)$ are local perturbations (supported on disks). Consequently, one can approximate an element $f\in \Ham^r_\omega(S)$ by an element of  $\cH^r_{\omega}(S)$. In other words, one can affirm that 
 $\cH^r_{\omega}(S)$ is residual in $\Ham^r_\omega(S)$.
 
By an immediate consequence of Lefschetz' formula, that will be explained in Section \ref{ss.lefschtez}, one knows that a diffeomorphism $f\in  \cG^r_{\omega}(S)$ isotopic to the identity has at least $2g-2$ hyperbolic fixed points with no reflection, with a strict inequality if some fixed points are elliptic or hyperbolic with reflection. More generally, every diffeomorphism $f\in  \cG^r_{\omega}(S)$ has at least $2g-2$ hyperbolic periodic points, with a strict inequality if some periodic points are elliptic (see \cite{LecSa} for details). Let us state now one of the main results of \cite{LecSa}.

\begin{theorem}\label{t.numberhyp} If $f\in \cG^r_{\omega}(S)$ has more than $2g-2$ periodic points, then every hyperbolic periodic point has a transverse homoclinic intersection.

\end{theorem}

 As a consequence of Arnold's conjecture, one knows that  every diffeomorphism $f\in\cH^r_{\omega}(S)$ has at least $2g+2$ fixed points (see Floer \cite{Fl} or Sikorav \cite{Si}). Consequently we have:

\begin{corollary}\label{c.numberhyphamiltonian}
If $f\in \cH^r_{\omega}(S)$, then every hyperbolic periodic point of $f$ has a transverse homoclinic intersection.\footnote{A similar result was announced by Xia \cite{X2}}
\end{corollary}

To obtain Theorem \ref{t.main} we will look more carefully at the homoclinic intersections appearing in Corollary \ref{c.numberhyphamiltonian}. Let us mention first some classical facts.

We denote by $\tilde S$ the universal covering space of $S$ and by $\tilde\pi:\tilde S\to S$ the covering projection. We denote by $G$ the group of covering automorphisms. For every $T\in G$, we write $[T]\in H_1(S,\Z)$ for the image of $T$ by the Hurewicz morphism: if $\tilde\gamma$ is a continuous path joining a point $\tilde z\in \tilde S$ to $T(\tilde z)$, then $[T]$ is the homology class of the cycle $\tilde\pi\circ\tilde\gamma$.

Suppose that $f$ is a homeomorphism of $S$ isotopic to the identity and consider an identity isotopy $I=(f_t)_{t\in[0,1]}$ of $f$. The isotopy $I$ can be lifted in $\tilde S$ to an identity isotopy of a certain lift $\tilde f_I$ of $f$. Denote $\mathrm{fix}_{I}(f)$ the set of fixed points of $f$ such that the loop $I(z):t\mapsto f_t(z)$ is homotopic to zero. It coincides with the set of fixed points of $f$ that are lifted to fixed points of $\tilde f_I$. In the case where $g\geq 2$, the map $f_I$ does not depend on $I$, it is called the {\it canonical lift} $\tilde f$ of $f$ and denoted by $\tilde f$. This map commutes with every $T\in G$. The fixed points of $f$ that are lifted to fixed points of $\tilde f$ are called {\it contractible fixed points} and the set of such points is denoted by $\mathrm{fix}_{\mathrm{cont}}(f)$.   In the case where $g=1$, for every fixed point $z$ of $f$, there exists an identity isotopy $I$ of $f$ such that $z\in\mathrm{fix}_{I}(f)$. In that situation, any lift of $f$ to $\tilde S$ commutes with every $T\in G$.

As it will explained in Section \ref{ss.lefschtez}, if $f\in \cG^r_{\omega}(S)$ and if $I$ is an identity isotopy of $I$, then $\mathrm{fix}_{I}(f)$ contains at least $2g-2$ hyperbolic fixed points with no reflection, and more in case it contains an elliptic or a hyperbolic fixed point with reflection. Remind also that if  $f\in \cH^r_{\omega}(S)$ and if $I$ is Hamiltonian, then $\mathrm{fix}_{I}(f)$ contains at least $2g+2$ points and among them at least $2g$ hyperbolic fixed points with no reflection. In this situation $\tilde f_I$ will be  the canonical lift of $f$ and the points of $\mathrm{fix}_{I}(f)$ the contractible fixed points.

The next result is very classical, related to the existence of a {\it rotational horseshoe} (see \cite{AroChHaMcg}, \cite{HocHol} or \cite{Lev}):
\begin{proposition}\label{p.rotationhorseshoe}
Consider $f\in \Diff^r_{\omega, *}(S)$ and an identity isotopy $I$ of $f$. Suppose that there exists a hyperbolic fixed point $z\in \mathrm{fix}_I(f)$ with no reflection, a lift $\tilde z\in \tilde S$ of $z$ and $T\in G\setminus\{\mathrm{Id}\}$, such that there exists an unstable branch of $\tilde z$ that intersects transversally a stable branch of $T(\tilde z)$. Then there exists $a>0$ and for every $p/q\in(0,a)$ written in an irreducible way, a  point $\tilde z_{p/q}$ such that $\tilde f_I^q(\tilde z_{p/q})= T^p(\tilde z_{p/q})$. 
 \footnote{The fact that $f$ preserves $\omega$ is not necessary in the statement}
\end{proposition}

\begin{remark*}It is easy to prove that $\tilde z_{p/q}$ projects onto a periodic point $z_{p/q}$  of $f$ of period $q$, because $p$ and $q$ are relatively prime. Denoting $O_{p/q}$ the orbit of $z_{p/q}$ it is also easy to prove that these orbits are all distinct. Observe now that $\mathrm{rot}_I(O_{p/q})=(p/q)[T]$. 
\end{remark*}

Let us state now the key result of the article:

\begin{proposition}\label{p.keyresult}
Consider $f\in \mathcal{G}^r_{\omega}(S)$ isotopic to the identity, an identity isotopy $I$ of $f$, and suppose that\begin{itemize} \item  $f$ has more than $2g-2$ periodic points if $g\geq 2$;
\item  $\mathrm{fix}_I(f)\not=\emptyset$ if $g=1$. 
\end{itemize} Then, there exists a hyperbolic fixed point $z\in\mathrm{fix}_I(f)$ with no reflection satisfying the following:  if $\tilde z\in \tilde S$ is a lift of $z$ and $\Lambda_s$, $\Lambda_u$ are branches of $\tilde z$, respectively stable and unstable, there exists $T\in G$ with $[T]\not=0$, such that  $\Lambda_u$ intersects transversally $T(\Lambda_s)$.\footnote{Of course in the case where $g\geq 2$, the proposition applies to the set $\mathrm{fix}_{\mathrm{cont}}(f)$. We need this more complicated formulation to include the case where $g=1$. It will be the case in the whole article} \end{proposition}

Theorem \ref{t.main} is an immediate consequence of Proposition  \ref{p.rotationhorseshoe} and Proposition \ref{p.keyresult}. Indeed, the set of maps $f\in\Ham^r_{\omega}(S)$ that satisfy the hypothesis of Proposition \ref{p.rotationhorseshoe}, with a Hamiltonian isotopy $I$, is open for the $C^1$ topology and consequently for the $C^r$-topology. It is dense for the $C^r$-topology because it contains $\cH^r_{\omega}(S)$. Indeed  every map $f\in\cH^r_{\omega}(S)$ satisfies the hypothesis of Proposition  \ref{p.keyresult}, with a Hamiltonian isotopy $I$. 

As explained in \cite{LecSa}, the set of $f\in \Diff^r_{\omega, *}(S)$ that have more than $2g-2$ hyperbolic periodic points, is open and dense in $\Diff^r_{\omega, *}(S)$ and so the condition to have infinitely many periodic orbits with non zero rotation vector is open and dense in $\Diff^r_{\omega, *}(S)$.

We can give a much more precise statement  of Proposition \ref{p.keyresult} but before we need to introduce some terminology that will be reviewed in detail in Section \ref{s.preliminaries}. We will say that two hyperbolic periodic points $z$ and $z'$ of a map $f\in \cG^r_{\omega}(S)$  are {\it equivalent}, and we will write $z\sim z'$, if the branches of $z$ and the branches of $z'$ have the same closure. One gets an equivalence relation on the set of hyperbolic periodic points.  We will denote $\cE(f)$ the set of equivalence classes.  A \textit{periodic regular domain} is a connected periodic  open set $V$ of finite type whose complement has no isolated points. For every $\ka\in \cE(f)$ we define a subspace $\iota_*(H_1(\kappa,\R))$ of $H_1(S,\R)$ by
$$\iota_*(H_1(\kappa,\R)):=\bigcap\{\iota_*(H_1(V,\R)): V \text{ is a periodic regular domain, }\ka\subset V\}$$
where $\iota:V\to S$ is the inclusion map (see Section \ref{ss:homoclinic classes}).

\begin{proposition}\label{p.precisekeyresult}
Consider $f\in \mathcal{G}^r_{\omega}(S)$ isotopic to the identity, an identity isotopy $I$ of $f$, and suppose that 
\begin{itemize}\item $f$ has more than $2g-2$ periodic points if $g\geq 2$;
\item  $\mathrm{fix}_I(f)\not=\emptyset$ if $g=1$. 
\end{itemize}
Then there exists a hyperbolic fixed point $z\in\mathrm{fix}_I(f)$ with no reflection whose class $\kappa\in {\mathcal E}(f)$ satisfies $\iota_*(H_1(\kappa,\R))\not=\{0\}$. In that case, if $\tilde z\in \tilde S$ is a lift of $z$ and $\Lambda_s$, $\Lambda_u$ are branches of $\tilde z$, respectively stable and unstable, there exist $T_1, T_2,\cdots, T_\ell\in G$ such that
\begin{enumerate}
\item $\{[T_1],\cdots,[T_\ell]\}$ is a base of  $\iota_*(H_1(\kappa,\R))$;
\item $\Lambda_u$ intersects transversally $T_i(\Lambda_s)$.
\end{enumerate}
\end{proposition}

It is an immediate consequence of the inclination lemma (see \cite{Pal}) and of the fact that $\tilde f_I$ commutes with the elements of $G$ that the set of $T\in G$ such that $\Gamma_u$ intersects transversally $T(\Gamma_s)$ is stable by composition. Consequently for every $(n_1,\dots, n_{\ell})\in\N^{\ell}\setminus\{0,\dots, 0)\}$ there exists $T\in G$ satisfying $[T]=\sum_{i=0}^{\ell} n_i[T_i]$ such that $\Lambda_u$ intersects transversally $T(\Lambda_s)$.  The above proposition implies that the rotation set of $f$ has nonempty interior within $\iota_*(H_1(\kappa,\R))$. In fact,
there exists $a>0$ such that for any $\textbf{p/q}=(p_1/q_1,\cdots, p_\ell/q_\ell)\in (0,a)^\ell$ where $p_i/q_i$ are written in an irreducilbe way, a periodic point $z_{\textbf{p/q}}$ exists with $\mathrm{rot}_I(O_{\textbf{p/q}})=\sum_{i=1}^{i=\ell}(p_i/q_i)[T_i]$ where $O_{\textbf{p/q}}$ is the orbit of $z_{\textbf{p/q}}$.  This can be done by finding a small rectangle $R$ containing $z$, such that for  some power $f^m$, setting $\tilde R$ the connected component of  $\pi^{-1}(R)$ containing $\tilde z$, we have that $\tilde f^m(\tilde R)$ intersects $T_i(\tilde R)$ in a Markovian way. This implies that polyhedron in $\iota_*(H_1(\kappa,\R))$ with vertices $0$ and $[T_i], 1\le i\le \ell$, is contained in the rotation set of $f^m.$ See \cite{AlBrPas} for an explicit construction in the case of an Axiom A diffeomorphim. 

We can justify Proposition \ref{p.precisekeyresult} with the following:

\begin{proposition}\label{p.dimensionrotation set}
Consider $f\in \Diff^r_{\omega, *}(S)$, an identity isotopy $I$ of $f$ and suppose that \begin{itemize} \item $f$ has more than $2g-2$ periodic points if $g\geq 2$;
\item  $\mathrm{fix}_I(f)\not=\emptyset$ if $g=1$. \end{itemize}
Then the subspace of $H_1(S,\R)$ generated by the $\iota_*(H_1(\kappa,\R))$, where $\kappa\in {\mathcal E}(f)$ contains a hyperbolic fixed point in $\mathrm{fix}_I(f)$ with no reflection, has dimension at least $g$ and contains a Lagrangian subspace (for the intersection form $\wedge$ defined on $H_1(S,\R)$). Consequently, the subspace of $H_1(S,\R)$ generated by the rotation set $\mathrm{rot}_I(f)$ satisfies the same property.
\end{proposition}

\begin{remark*}  

If $r\geq 3$ then there exists an open set $\mathcal O$ of $\mathrm{Ham}^r_{\omega}(S)$ such that for every $f\in\mathcal O$, the subspace  of $H_1(S,\R)$ generated by $\mathrm{rot}_I(f)$ has dimension $\leq g$, if $I$ is a Hamiltonian isotopy.  It is a consequence of KAM theory. Indeed, consider a sequence $(A_i)_{1\leq i\leq g}$ of  disjoint closed annuli whose union has a connected complement of genus $0$. Fix $a>0$ and consider a map $f_*\in\mathrm{Ham}^{\infty}_{\omega}(S)$ that is conjugate to the classical integrable twist $\tau:\R/\Z\times [0,a]\to \R/\Z\times [0,a]$ on $A_i$ by a diffeomorphism $h_i$ such that $h_i^*(dx\wedge dy)=\omega$. It means that $\tau$ is lifted by $\tilde\tau:\R\times [0,a]\to \R\times [0,a]$, where $\tilde\tau(x,y)=(x+y,y)$. By a result of Herman \cite{He}, every map $f\in\mathrm{Ham}^3_{\omega}(S)$  close to $f_*$, for the $C^3$-topology, admits, for every $i\in\{1,\dots, g\}$, at least one invariant loop $\Gamma_i$ that is not null homotopic in $A_i$. Then,  for every  Borel probability measure $\mu$ invariant by $f$, it holds that $\rho_I(\mu)\wedge [\Gamma_i]=0$. By Propositions \ref{p.precisekeyresult} and \ref{p.dimensionrotation set}
we deduce that  there exists a non empty open set $\mathcal O'\subset \mathcal O$ of $\mathrm{Ham}^r_{\omega}(S)$ such that for every $f\in\mathcal O'$, the subspace  of $H_1(S,\R)$ generated by $\mathrm{rot}_I(f)$ has dimension $g$, if $I$ is a Hamiltonian isotopy.

Let us consider now the case where $r=1$. According to Arnaud-Bonatti-Crovisier \cite{ArnBoCr}, there exists a residual set of ${\mathcal H}^1_{\omega}(S)$ whose elements have a dense orbit and a unique class $\kappa$. It is natural question to ask whether  there exists a residual set of ${\mathcal H}^1_{\omega}(S)$ whose elements still have a dense orbit (and a unique class) when lifted to a finite covering. We will see later that it is equivalent to say that  $\iota_*(H_1(\kappa,\R))=H_1(S,\R)$ or to say that $\mathrm{gen}(\kappa)=2g$. In that case, the subspace of $H_1(S,\R)$ generated by the rotation set $\mathrm{rot}_I(f)$
has dimension 2g. It is not difficult to prove that if it is true, then generically in  ${\mathcal H}^1_{\omega}(S)$, there is no periodic regular domain but $S$. Moreover, a positive answer to the question would imply that  there exists an open and dense set $\mathcal O'$ of $\mathrm{Ham}^1_{\omega}(S)$ such that for every $f\in\mathcal O'$, the subspace  of $H_1(S,\R)$ generated by $\mathrm{rot}_I(f)$ has dimension $2g$, if $I$ is a Hamiltonian isotopy.

\end{remark*}
\subsection{Organization of the  article and acknowledgements}\label{ss.organization} In Section \ref{s.preliminaries} we will state and develop some concepts, results and tools which are basics along the paper. More precisely, in  Section \ref{ss.lefschtez}, we will remind classical results about Lefschetz formula and in Sections \ref{ss.regular} and \ref{ss:homoclinic classes} results obtained in \cite{LecSa} concerning generic conservative diffeomorphisms and related to periodic regular domains and classes of hyperbolic periodic points, in parti-cular when these classes are homoclinic, meaning that the stable and unstable branches of the elements of the class intersect. The results appearing in Sections \ref{ss:Minimal regular neighborhoods} and  \ref{ss.moreaboutregulardomains} are new: we will introduce the notion of minimal regular neighborhood of a class and will state a Lefchetz formula for such an object. The proofs of Propositions \ref{p.keyresult} and \ref{p.precisekeyresult} will be given in Section \ref{s.keyresult}. We will end the article with Section \ref{s.poincarebirkhofflike}  where we will explain how to apply the classical Poincar\'e-Birkhoff theorem to get a criterium of existence of infinitely many periodic orbits with non zero rotation number.

We would like to thank Viktor Ginzburg and Ba\c{s}ak G\" urel for suggesting us to look at this problem. We would also like to thank Sylvain Crovisier and Raphael Krikorian for useful comments. 

\section{Preliminaries}\label{s.preliminaries}

\subsection{Lefschetz index}\label{ss.lefschtez}

Lefschetz formula says that for every homeomorphism $f$ of $S$ with finitely many fixed points, we have
$$\sum_{z\in \mathrm{fix}(f)}i(f,z)=\sum_{i=0}^2(-1)^i\mathrm{tr}(f_{*,i}),$$
where $i(f,z)$ is the Lefschetz index of $f$ at $z$ and $f_{*,i}$  is the endomorphism of the $i$-th homology group $H_i(S,\R)$ induced by $f.$ In particular, if $f$ is isotopic to the identity, we have
$$\sum_{z\in \mathrm{fix}(f)}i(f,z)=\chi(S),$$
where $\chi(S)=2-2g$ is the Euler characteristic of $S$. In that case, there is a more precise version of Lefschetz formula. Suppose that $f$ is a homeomorphism of $S$ isotopic to the identity and consider an identity isotopy $I$ of $f$. Then it holds that 
$$\sum_{z\in \mathrm{fix}_I(f)}i(f,z)=\chi(S).$$

Suppose that $f\in\mathcal{G}^r_\omega(S)$. Then we have:
\begin{itemize}
\item $i(f,z)=-1$ if $z$ is a hyperbolic fixed point with no reflection,
\item  $i(f,z)=1$ if $z$ is a hyperbolic fixed point with reflection,
\item $i(f,z)=1$ if $z$ is an elliptic fixed point.
\end{itemize}

Denoting $\mathrm{fix}_\mathrm{h^+}(f)$ the set of  hyperbolic fixed points with no reflection, it holds that  $$\#\mathrm{fix}_\mathrm{h^+}(f)\ge- \sum_{z\in \mathrm{fix}(f)}i(f,z)=-\chi(S)=2g-2,$$
with a strict inequality in the presence of an elliptic fixed point or a hyperbolic fixed point with reflection. 
Similarly, if $I$ is an identity isotopy of $f$, we have
$$\#\mathrm{fix}_{I,\mathrm{h^+}(f)}\ge 2g-2,$$ 
where $\mathrm{fix}_{I,\mathrm{h^+}}(f)=\mathrm{fix}_I(f)\cap\mathrm{fix}_\mathrm{h^+}(f)$.

\subsection{Regular domains and generic conservative diffeomorphisms}\label{ss.regular}

A {\it regular domain} of $S$ is a connected open set $V$ of finite type whose complement has no isolated point. Such a set has finitely many ends and its complement has finitely many connected components, none of them reduced to a point. One proves easily that every connected component of the intersection of finitely many regular domains is a regular domain (see \cite{LecSa}). If $V$ is a regular domain, it can be compactified in three natural ways.
\begin{itemize}

\item The {\it ambient compactification} is the closure $\overline V$ of $V$ in $S$.

\item The {\it end compactification} is obtained by adding every end of $V$: one gets a boundaryless compact surface $\check V$.
 \item
The {\it prime end compactification} is obtained by blowing up every end by the circle of prime ends (see Mather \cite{Math}): one gets a compact surface with boundary $\widehat V$.

\end{itemize}

If $V$ is invariant by an orientation preserving homeomorphism $f$ of $S$, then $f_{\vert \overline V}$ is an extension of $f_{\vert V}$ to $\overline V$.  Moreover there exists a natural extension $\check f$ of $f_{\vert V}$ to $\check V$ that permutes the ends. An important property of the prime end compactification is that $f_{\vert V}$ admits an extension by a homeomorphism $\widehat f$ of $\widehat V$: every added circle $C$ is periodic and if $q$ is its period, then $\widehat{f}^{q}{}_{\vert C}$ is orientation preserving, so the rotation number $\mathrm{rot}(C)\in\R/\Z$ of $\widehat f^{q}{}_{\vert C}$ can be defined if $C$ is endowed with the induced orientation.

\medskip

Let us state an important property, easy to prove, satisfied by a map $f\in \cG^r_{\omega}(S)$: if $C$ is a simple loop invariant by a power of $f\in \cG^r_{\omega}(S)$, then there is no periodic point of $f$ on $C$.  The fundamental following generalization is due to Mather \cite{Math}:

\begin{theorem}\label{t.pe-genericconsequences}
Consider $f\in \cG^r_{\omega}(S)$. If $V$ is a regular domain invariant by $f$, then $\widehat f$ has no periodic point on the boundary of $\widehat V$. Equivalently, for every added circle $C$, one has $\mathrm{rot}(C)\not\in \Q/\Z$.
\end{theorem}
The proof of the theorem was using a slightly different condition than (G3)\footnote{We can replace (G3) by the condition that  every elliptic periodic point is surrounded by periodic curves of irrational rotation number. Such a condition would be generic for $r$ large enough. In fact, due to a recent result of Contreras and Oliveira \cite{CoO}, it is highly probable that to do define $\mathcal G_{\omega}^r(S)$, no condition is needed about elliptic periodic points except the fact that the eigenvalues of the derivatives are not roots of the unity.} but was extended to our situation in \cite{KLecN}. In the same article the following was shown (see  \cite[Theorem E]{KLecN}):

\begin{theorem}\label{t.genericconsequences}
Consider $f\in \cG^r_{\omega}(S)$. If $V$ is a regular domain invariant by a power of $f\in \cG^r_{\omega}(S)$, then $f$ has no periodic point on the frontier of $V$ in $S$.
\end{theorem}

An immediate consequence of Theorem \ref{t.pe-genericconsequences} is the following:

\begin{corollary}\label{c.e-genericconsequences}
Consider $f\in \cG^r_{\omega}(S)$. If $V$ is a regular open set invariant by $f\in \cG^r_{\omega}(S)$ and if $\zeta$ is an end of $V$ of period $q$ (as a periodic point of $\check f)$, then for every $n\geq 1$, the Lefschetz index of   $\check f^{qn}$ at $\zeta$ is equal to $1$.
\end{corollary}

Let us conclude this subsection with the following result, proved by Mather \cite{Math} and extended in \cite{X1} and \cite[Corollary 8.9]{KLecN}.
\begin{theorem}\label{t.genericconsequences2} Consider $f\in \cG^r_{\omega}(S)$. The four branches of a hyperbolic periodic point $z$ of $f$ accumulate on $z$ and
 have the same closure in $S$.
\end{theorem}

\subsection{Homoclinic classes}\label{ss:homoclinic classes}

Consider $f\in \cG^r_{\omega}(S)$. We will say that two hyperbolic periodic points $z$ and $z'$ of
$f$  are {\it equivalent} if the branches of $z$ and the branches of $z'$ have the same closure. One gets an equivalence relation on the set of hyperbolic periodic points.  We will denote $\cE(f)$ the set of equivalence classes and for every $\kappa\in\cE(f)$, we will write $K(\kappa)$ for the closure of a branch of an element $z\in\kappa$. The branches of elements $z\in\kappa$ will be called {\it branches of $\kappa$}. The map $f$ acts naturally on $\cE(f)$ as a bijection and every orbit is finite (because $f^q(\kappa)=\kappa$ if $\kappa$ contains a fixed point of $f^q$) and so one can define the period of $\kappa$ as the cardinal of its orbit. Of course $\cE(f^q)=\cE(f)$, for every $q\geq 2$.

\medskip

We will remind some results about classes  and regular domains that are proved in \cite{LecSa}. The proofs generally use Theorem \ref{t.genericconsequences} and Theorem \ref{t.genericconsequences2}.

\begin{proposition}\label{p.regulardomain-class}
Consider $f\in \cG^r_{\omega}(S)$. If $V$ is a periodic regular domain of $f$, then every class $\kappa\in \cE(f)$ is included in $V$ or disjoint from $\overline V$. In the first situation the branches of $\kappa$ are all included in $V$, in the second situation they are all disjoint from $\overline V$.\end{proposition}

\begin{corollary}\label{c.connectedcomponent}
Consider $f\in \cG^r_{\omega}(S)$. If $V$, $V'$ are periodic regular domains of $f$ that both contain a class $\kappa\in \cE(f)$, then $\kappa$ is included in a connected component $V''$ of $V\cap V'$ and $V''$ is a periodic regular domain.\end{corollary}

\begin{proposition}\label{p.genericconsequences}
Fix $f\in \cG^r_{\omega}(S)$, $\kappa\in \cE(f)$ and $z'\in\mathrm{per}_h(f)$. If $z'$ belongs to $K(\kappa)$,  then it belongs to $\kappa$. \end{proposition}

The next result is a very slight improvment of Corollary 3.3 of \cite{LecSa}, with a similar proof . 

\begin{proposition}\label{p.separationconsequences}
Consider $f\in \cG^r_{\omega}(S)$. Let $(\kappa_i)_{1\leq i\leq p}$ be a sequence of different classes in ${\mathcal E}(f)$ and $(V_i)_{1\leq i\leq p}$ a sequence of periodic regular domains of $f$, such that $\kappa_i\subset V_i$, if $1\leq i\leq p$. Then one can find a family $(V'_i)_{1\leq i\leq p}$ of pairwise disjoint periodic regular domains such that $\kappa_i\subset V'_i\subset V_i$ if $ 1\leq i\leq p$.\end{proposition}

\begin{proof}  Fix $1\leq i<j\leq p$. Every connected component of $S\setminus K(\kappa_j)$ is a periodic regular domain and we know by
Proposition  \ref{p.genericconsequences} that $\kappa_i\cap K(\kappa_j)=\emptyset$. So, by Proposition \ref{p.regulardomain-class}, there exists a connected component $W_i^j$ of $S\setminus K(\kappa_j)$ that contains $\kappa_i$. We define inductively a family $(W_i)_{1\leq i\leq p}$ of pairwise disjoint periodic regular domains such that $\kappa_i\subset W_i$ in the following way;
\begin{itemize}
\item  $W_1$ is the connected component of $\bigcap_{1< j\leq p} W_1^j$ that contains $\kappa_1$:
\item for every $i>1$,  $W_i$ is the connected component of $$ \left( \bigcap_{i<j\leq p} W_i^j\right)\setminus \left(\bigcup_{1\leq j<i} \overline{W_j}\right)$$ that contains $\kappa_i$.
\end{itemize}
To finish the proof, it is sufficient to define $V'_i$ has being the connected component of $V_i\cap W_i$ that contains $\kappa_i$.
\end{proof}

Remind that the genus of an open set $V\subset S$, denoted $\mathrm{gen}(V)$, is the largest integer $s$ such that we can find a family of simple loops $(\lambda_i)_{0\leq i<2s}$ satisfying:
\begin{itemize}
\item $\lambda_{2j}$ and $\lambda_{2j+1}$ intersect in a unique point;
\item $[\lambda_{2j}]\wedge[\lambda_{2j+1} ]=1$;
\item $\lambda_i\cap \lambda_{i'}=\emptyset$, if $i\not=i'$ and $\{i,i'\}$ is not a set $\{2j, 2j+1\}$.
\end{itemize}
 Let us define now the {\it genus of a class $\kappa \in\cE(f)$}, where $f\in\cG^r_{\omega}(S)$, as being the integer $\mathrm{gen}(\kappa)\in\{0,\dots, g\}$ uniquely defined by the following conditions:
\begin{enumerate}
\item $\kappa$ is contained in a periodic regular domain of genus $\mathrm{gen}(\kappa)$;
\item  $\kappa$ is not contained in a periodic regular domain of genus $<\mathrm{gen}(\kappa)$.
\end{enumerate}

The function $\kappa\mapsto \mathrm{gen}(\kappa)$ satisfies some additive properties. For example, in the statement of Proposition \ref{p.separationconsequences}, it holds:
$$\sum_{i=1}^p\mathrm{gen}(\kappa_i)\leq \sum_{i=1}^p\mathrm{gen}(V_i)\leq \mathrm{gen}(V).$$

 Let us conclude with the following proposition, which is the fundamental result of \cite{LecSa}:

\begin{proposition}\label{p.localLefschetz}
Consider $f\in \cG^r_{\omega}(S)$. If $f$ has more than $2g-2$ periodic points. Then
\begin{itemize}
\item every class $\kappa\in \cE(f)$ is infinite;
\item if $z$ and $z'$ are equivalent hyperbolic periodic points (possibly equal), then every stable branch of $z$ meets every unstable branch of $z'$.
\end{itemize}
\end{proposition}

In this situation, we will say that $f$ has {\it  homoclinic classes}. In that case, two hyperbolic periodic points $z$, $z'$ are equivalent if and only if a stable branch of $z$ meets an unstable branch of $z'$. We can classify the diffeomorphisms $f\in\cG^r_{\omega}(S)$ which have not homoclinic classes. First, every periodic point is hyperbolic. Moreover in the case where $S$ is the $2$-torus, $f$ is periodic point free and must be isotopic either to a Dehn twist map or to the identity. In the case where $g\geq 2$, there is a power of $f$ that is isotopic to the identity. The classical example is given by the time one map of the flow in a minimal direction for a translation surface in the principal stratum (see \cite{LecSa} or \cite{Lec}). 

The following simple fact will be used in the article: if $\pi:S'\to S$ is a finite covering and $f'$ is a lift to $S'$ of a diffeomorphism $f$ of $S$, then $f\in \cG^r_{\omega}(S)$ if and only if $f'\in\cG^r_{\pi^*(\omega)}(S')$. Moreover $f$ has homoclinic classes if and only if $f'$ has homoclinic classes.

\subsection{Minimal regular neighborhoods}\label{ss:Minimal regular neighborhoods}

The notion of genus of a class was important in \cite{LecSa}. In the present article, we will use a finer object.  If $V$ is an open set of $S$, we will consider $\iota_*(H_1(V,\R))\subset H_1(S,\R)$, where $\iota: V\to S$ is the inclusion map.
Fix $f\in \cG^r_{\omega}(S)$ and $\kappa\in{\mathcal E}(f)$. Using Corollary \ref{c.connectedcomponent}, one can define a subspace $\iota_*(H_1(\kappa,\R))$ of $H_1(S,\R)$ by the following conditions:
\begin{itemize}
\item there exists a periodic regular domain $V$ containing $\kappa$ such that $\iota_*(H_1(V,\R))=\iota_*(H_1(\kappa,\R))$;
\item for every periodic regular domain $V$ containing $\kappa$, it holds that $\iota_*(H_1(\kappa,\R))\subset \iota_*(H_1(V,\R))$.
\end{itemize}
Let us introduce now a notion that will be useful later. Let $V$ be a regular domain of $S$. We will say that $\Sigma\subset V$ is a {\it compact approximation} of $V$ if $\Sigma$ is a non empty compact surface such that every connected component of $V\setminus\Sigma$ is a topological open annulus, meaning that it is homeomorphic to $\R/\Z\times\R$. Note that $\Sigma$ is a surface with boundary except in the case where $\Sigma=V=S$. If $A$ is a connected component of $V\setminus\Sigma$, its frontier in $V$ is a loop $\Gamma$, one of the connected components of $\partial \Sigma$. The set $A$ is a punctured neighborhood of an end of $V$ that will be denoted $e(\Gamma)$. In the case where the connected component of $S\setminus\Sigma$ that contains $A$ is an open disk $D$, we will say that $\Gamma$ {\it bounds a disk outside} $\Sigma$. In that case, $D\setminus A$ is a connected component of $S\setminus V$ and this component is {\it cellular}, meaning the intersection of a decreasing sequence of topological closed disks. We will say that the end $e(\Gamma)$ is cellular. Note that every cellular connected component of $S\setminus V$ is obtained in that way.  
 Note also that $\chi(V)=\chi(\Sigma)$ and that two compact approximations $\Sigma$, $\Sigma'$ of a regular domain are isotopic: there exists a continuous family of homeomorphisms $(h_t)_{t\in[0,1]}$ supported in $V$ such that $h_0=\mathrm{Id}$ and such that $h_1(\Sigma)=\Sigma'$.

Let us explain why $\mathrm{gen}(\kappa)$ is equal to half the rank of the form $\wedge_{\vert {\iota_*(H_1(\kappa,\R))}}$. It will implies that the following inclusion holds: $$\iota_*(H_1(V,\R))= \iota_*(H_1(\kappa,\R)) \Rightarrow \mathrm{gen}(V)=\mathrm{gen}(\kappa),$$ if $V$ is a regular periodic open set containing $\kappa$. Indeed, if $V$ is a regular open set and if $\Sigma\subset V$ is a compact approximation of $V$,  the subspace $H$ of $H_1(S,\R)$ generated by the classes $[\Gamma]$, where $\Gamma$ is a boundary curve of a $\Sigma$, is included in $\iota_*(H_1(V,\R))$. In fact it coincides with the kernel of  $\wedge_{\vert {\iota_*(H_1(V,\R))}}$. Indeed, it is obviously included in the kernel. Moreover $\iota_*(H_1(V,\R))/H$ is isomorphic to $H_1(\check V,\R)$ and the intersection form is non degenerated on $H_1(\check V,\R)$ because $\check V$ is a closed surface. As we now that $\mathrm{dim}(H_1(\check V,\R))=2\mathrm{gen}(\check V)= 2\mathrm{gen}(V)$, we deduce that the rank of  $\wedge_{\vert {\iota_*(H_1(V,\R))}}$ is twice the genus of $V$.

Let us conclude this section with a notion that is not explicitly defined in \cite{LecSa} but will be useful, mainly to lighten the redaction. Corollary \ref{c.connectedcomponent} is used several times in what follows. If $\kappa$ is a class of period $q$,  define a {\it regular neighborhood} of $\kappa$ to be a regular domain $V$ invariant by $f^q$ that contains $\kappa$ and such that the $f^k(V)$, $0\leq k<q$, are pairwise disjoint.

\begin{proposition}\label{p.minimal}
Consider $f\in \cG^r_{\omega}(S)$ and $\kappa\in{\mathcal E}(f)$. For every regular periodic domain  $V$ containing $\kappa$, there exists a regular neighborhood $V''$ of $\kappa$ such that $V''\subset V$. Moreover if $V$, $V'$ are two regular neighborhoods of $\kappa$, then $\kappa$ is included in a connected component $V''$ of $V\cap V'$ and $V''$ is a regular neighborhood of $\kappa$.\end{proposition}
\begin{proof}
By Proposition \ref{p.separationconsequences}, one can find a sequence of disjoint periodic re-gular domains  $(V_k)_{0\leq k<q}$ such that $f^k(\kappa)\subset V_k\subset f^k(V)$. By Corollary \ref{c.connectedcomponent}, $\kappa$ is included in a connected component $V'$ of $\bigcap_{0\leq k<q}f^{-k}(V_k)$ and $V'$ is a periodic regular domain included in $V$. Moreover it holds that $f^k(V')\cap V'=\emptyset$ if $0<k<q$ because we have $f^k(V')\cap V'\subset V_k\cap V_0=\emptyset$. Denote $q'$ the period of $V'$. Every $f^{kq}(V')$ is a regular domain that contains $\kappa$, and $W'= \bigcap_{0\leq k<q'} f^{kq}(V')$ satisfies $f^{q}(W')=W'$. By Corollary \ref{c.connectedcomponent}, $\kappa$ is included in a connected component $V''$ of $W'$ that contains $\kappa$ and $V''$ is a periodic regular domain. Note that $f^q(V'')=V''$ because $f^q(W')=W'$ and $f^q(\kappa)=\kappa$.  Moreover we have $f^k(V'')\cap V''=\emptyset$ if $0<k<q$, because $V''\subset V'$. Consequently, $V''$ is a regular neighborhood of $\kappa$. It is included in $V$ because we have $V''\subset W'\subset V'\subset V$. 

Using again Corollary \ref{c.connectedcomponent}, we know that if $V$ and $V'$ are regular neighborhoods of $\kappa$, then $\kappa$ is included in a connected component $V''$ of $V\cap V'$ and $V''$ is a periodic regular domain. We have $f^q(V'')=V''$ because $f^q(V\cap V')=V\cap V'$ and $f^q(\kappa)=\kappa$. We know that $f^k(V'')\cap V''=\emptyset$ if $0<k<q$, because  $f^k(V)\cap V=f^k(V')\cap V'=\emptyset$. So, $V''$ is a regular neighborhood of $\kappa$.
\end{proof}

One deduces from Proposition \ref{p.minimal} that for every $f\in \cG^r_{\omega}(S)$ and every $\kappa\in{\mathcal E}(f)$, there exists a regular neighborhood $V$ of $\kappa$ such that
$$ \iota_*(H_1(V,\R))=\iota_*(H_1(\kappa,\R)), \enskip \mathrm{gen}(V)=\mathrm{gen}(\kappa).$$
Such a domain $V$ will be called a {\it minimal regular neighborhood} of $\kappa$.

\subsection{Lefschetz formula for regular domains}\label{ss.moreaboutregulardomains}

We will prove in this section the following result.

\begin{lemma} \label{l: lefschetzformula} Consider $f\in \cG^r_{\omega}(S)$ isotopic to the identity and an identity isotopy $I$ of $f$.  Suppose that $V$ is a regular domain with no cellular end that is invariant by $f$ and such that every connected component of $\tilde \pi^{-1}(V)$ is invariant by $\tilde f_I$. Then we have
 $$\sum_{z\in \mathrm{fix}_{I}(f)\cap V} i(f,z)=\chi(V).$$
 \end{lemma}

\begin{proof}  The fact that $V$ has no cellular end implies that every connected component $\tilde V$ of $\tilde\pi^{-1}(V)$ is simply connected. Indeed, the complement of $\tilde V$ in the Alexandroff compactification of $\tilde S$ is connected. So, the map $\tilde \pi_{\vert \tilde V}: \tilde V\to V$ is the universal covering projection and one can write $V=\tilde V/H$, where $H\subset G$ is the stabilizer of $\tilde V$. The fact that $\tilde f_{I\vert \tilde V}$ commutes with every element of $H$ implies that $f_{\vert V}$ is isotopic to the identity.  Moreover, one can find an identity isotopy $I'$ of $f_{\vert V}$ that can be lifted to an identity isotopy of $\tilde f_{I\vert \tilde V}$ (this last remark is relevant only in the case where $V$ is an annulus). The isotopy $I'$ can be extended to an identity isotopy $\check I'$ of $\check f$ on $\check V$, that fixes the ends (remind that the end compactification of $V$ is denoted by $\check V$ and the natural extension of $f_{\vert V}$ to $\check V$ is denoted by $\check f$).   Consequently it holds that 
$$\mathrm{fix}_{\check I'}(\check f)=\left(\mathrm{fix}_{I}(f)\cap V\right) \cup \mathrm{end}(V),$$
where $\mathrm{end}(V)$ is the set of ends of $V$.  As stated in Corollary \ref{c.e-genericconsequences}, one knows that $ i(\check f,z)=1$ if $z$ is an end. Applying the improved version of Lefschetz formula,
one gets:
$$ \chi(V)=\chi (\check V)-\# \mathrm{end}(V) 
=\sum_{z\in \mathrm{fix}_{\check I'}(\check f)} i(\check f,z)-\# \mathrm{end}(V)
 =\sum_{z\in \mathrm{fix}_{I}(f)\cap V} i(f,z).$$
 
\end{proof}

\section{Proof of Propositions \ref{p.keyresult} and \ref{p.precisekeyresult}}\label{s.keyresult}

Let us begin by some lemmas:

\begin{lemma}  \label{l:nonzerohomology} Consider $f\in \cG^r_{\omega}(S)$ isotopic to the identity and an identity isotopy $I$ of $f$. Suppose that $\mathrm{fix}_I(f)\not=\emptyset$.\footnote{As explained before  this hypothesis is always satisfied if $g\geq 2$} If $(V_j)_{j\in J}$ is a finite family of disjoint regular domains invariant by $f$, such that $\mathrm{fix}_{I,\mathrm{h^+}}(f)\subset\bigcup_{j\in J} V_j$, then there exists $j\in J$ such that $\iota_*(H_1(V_j,\R))\not=\{0\}$. 
\end{lemma}
  
\begin{proof} Fix $j_0\in J$. The set of cellulars ends of $V_{j_0}$ is finite and invariant by $f_{\vert V_{j_0}}$. If one adds to $V_{j_0}$ the union of the cellular components of $S\setminus V_{j_0}$, one gets another  regular domain $W_{j_0}$ such that $\iota_*(H_1(W_{j_0},\R))=\iota_*(H_1(V_{j_0},\R))$ and $W_{j_0}$ is invariant by $f$. The set $W_{j_0}$ possibly contains other components $V_j$, $j\not=j_0$. As a consequence, reducing the set $J$ if necessary, there is no loss of generality by adding the hypothesis that no domain $V_j$, $j\in J$, has a cellular end. Of course, one can also suppose that for every $j\in J$ it holds that $\mathrm{fix}_{I,\mathrm{h^+}}(f)\cap V_j\not=\emptyset$, which implies that every connected component of $\tilde\pi^{-1}(V_j)$ is invariant by  $\tilde f_I$. Consequently, one can apply Lemma  \ref{l: lefschetzformula} to $V_j$.
For every $j\in J$, fix a compact approximation $\Sigma_j$ of $V_j$. By hypothesis, if $\Gamma$ is a component of $\partial\Sigma_j$, it does not bound a disk outside $\Sigma_j$.  One has
 $$\overline{S\setminus\bigcup_{j\in J} \Sigma _j}=\bigcup_{j'\in J'}\Sigma'_{j'},$$
 where $(\Sigma'_{j'})_{j'\in J'}$ is a family of disjoint compact subsurfaces with boundaries.
 We have
 $$\begin{aligned}\sum_{j'\in J'} \chi (\Sigma'_{j'})&= \chi(S)- \sum_{j\in J} \chi (\Sigma_j)\\
 &= \chi(S)- \sum_{j\in J} \chi (V_j)\\
 &=\sum_{z\in \mathrm{fix}_I(f)} i(f,z)-  \sum_{j\in J} \sum_{z\in \mathrm{fix}_I(f)\cap V_j} i(f,z).\end{aligned}$$
We deduce that
 $$\sum_{j'\in J'} \chi (\Sigma'_{j'})\geq 0,$$
 because every fixed point $z\in \mathrm{fix}_I(f)$  of Lefschetz index $-1$ is contained in a $V_j$. 

Note now that $\chi(\Sigma'_{j'})\leq 0$. Indeed, $\chi(\Sigma'_{j'})$ is positive if and only $\Sigma'_{j'}$ is a sphere or a disk. The first case is impossible because $S$ is not the $2$-sphere but the second case never occurs because no component of the boundary of a $\Sigma_j$ bounds a disk outside $\Sigma_j$. One deduces that $\chi(\Sigma'_{j'})=0$ for every $j'\in J'$, meaning that every $\Sigma'_{j'}$ is an annulus.

To get the lemma, it is sufficient to prove that at least one of the following situations occurs:

\begin{itemize}
\item there exists $j\in J$ such that $\mathrm{gen}(V_j)\not=0$,
\item there exists $j\in J$ and a component $\Gamma$ of the boundary of $\Sigma_j$ such that $[\Gamma]\not=\{0\}$.

\end{itemize}

If none of the situations occurs, then every $\Sigma_j$ is a surface of genus zero  with separating boundary components and every $\Sigma'_{j'}$ is an annulus. It implies that $S$ is the $2$-sphere. We have found a contradiction.  
\end{proof}

We can clarify the previous result.

\begin{lemma}  \label{l:completehomology} Consider $f\in \cG^r_{\omega}(S)$ isotopic to the identity and an identity isotopy $I$ of $f$. Suppose that $\mathrm{fix}_I(f)\not=\emptyset$. Let $(V_j)_{j\in J}$ be a finite family of disjoint regular domains invariant by $f$, such that 
\begin{itemize}
\item no domain $V_j$, $j\in J$, has a cellular end;
\item for every $j\in J$ it holds that $\mathrm{fix}_{I,\mathrm{h^+}}(f)\cap V_j\not=\emptyset$; 
\end{itemize}
and fix for every $j\in J$, a compact approximation $\Sigma_j$ of $V_j$. If $H$ is the subspace of $H_1(S,\R)$ generated by the classes $[\Gamma]$, where $\Gamma$ is a boundary curve of a $\Sigma_j$, and $H'$ is the subspace of $H_1(S,\R)$ generated  by the $i_*(H_1(V_j,\R))$, $j\in J$, then the equalities
$$\mathrm{dim}( H)+\mathrm{dim} (H')=2g, \enskip \sum_{j\in J} \mathrm{gen}(V_i) +\mathrm{dim} (H)=g $$ hold.
\end{lemma}
  
\begin{proof} We denote by $\mathcal C$ the set of simple loops that are boundary curves of a $\Sigma_j$, $j\in J$. We denote by $h$ the dimension of $H$, by $h'$ the dimension of $H'$, by $g_j$ the genus of $V_j$ and by $s_j$ the number of ends of $V_j$ (or equivalently the number of boundary curves of $\Sigma_j$).
We keep the same notations as in Lemma   \ref{l:nonzerohomology}. We have seen that every surface $\Sigma'_{j'}$ is an annulus. We have
$$2-2g= \chi(S)=\sum_{j\in J} \chi(\Sigma_j)+\sum_{j'\in J'} \chi(\Sigma'_{j'})=\sum_{j\in J} 2-2g_j-s_j,$$
which implies that 
$$g-\sum_{j\in J} g_j=1+ \frac{1}{2}\#{\mathcal C}-\sharp J.$$
To get the second equality of the lemma, one needs to prove that 
$$\mathrm{dim}(H)=  1+ \frac{1}{2}\#{\mathcal C}-\sharp J.$$
One can construct a new surface $S'$ homeomorphic to $S$ by identifying each annulus $\Sigma'_{j'}$ to a simple loop. We are reduced to prove that if $(\Gamma_k)_{1\leq k\leq m}$ is a family of pairwise disjoint simple loops, the dimension of the subspace of $H_1(S,\R)$ generated by the $[\Gamma_k]$, $1\leq k\leq m$, is equal to $1+m- n_m$, where $n_m$ is the number of connected component  of $S\setminus\cup_{1\leq k\leq m} \Gamma_k$. We argue by induction. The property is true if $m=1$: a simple loop is homological to zero if and only if it separates $S$ in two connected components. We suppose that the property is true for $m$ and we want to prove it for $m+1$. We denote by $V$ the connected component of $S\setminus \cup_{1\leq k\leq m} \Gamma_k$ that contains $\Gamma_{m+1}$. If it does not separate $V$, we can find a simple loop $\Gamma'\subset V$ such that $[\Gamma']\wedge [\Gamma_{m+1}]=1$. We deduce that $[\Gamma_{m+1}]$ is not a linear combination of the $\Gamma_k$, $1\leq k\leq m$. If it separates $V'$, then the closure of a connected component of $V\setminus\Gamma_{m+1}$ is a surface whose boundary consists of $\Gamma_{m+1}$ and some others $\Gamma_k$, $1\leq k\leq m$, and so  $[\Gamma_{m+1}]$ is a linear combination of the $[\Gamma_k]$, $1\leq k\leq m$. In both cases, the formula holds.

To get the first equality of the lemma, just notice that $H$ is included in $H'$ and that $\mathrm{dim} (H')=\mathrm{dim} (H)+2\sum_{j\in J} g_j$.
\end{proof}

 We deduce from Lemma  \ref{l:nonzerohomology} the following:

\begin{corollary}  \label{c:isotropic} Suppose that the hypothesis of  Lemma \ref{l:completehomology} are satisfied. Then the space $H'$ has dimension at least $g$. Moreover it contains a Lagrangian subspace of $H_1(S,\R)$. 

\end{corollary}

\begin{proof} The space $H$ is obviously isotropic, meaning that the form $\wedge$ vanishes on it. It implies that $\mathrm{dim}(H)\leq g$. By Lemma \ref{l:completehomology}  we deduce that $\mathrm{dim}(H')\geq g$. Denoting $h=\mathrm{dim}(H)\leq g$, we can find a family of simple loops $\Gamma_k$, $1\leq k\leq h$, each of them being a boundary  curve of a $V_j$, such that the family $([\Gamma])_{1\leq k\leq h}$ is basis of $H$. Recall that the genus of $V_j$ is denoted by $g_j$. One can find two families of simple loops $(\Gamma'_{j, l})_{1\leq l\leq g_j}$ and $(\Gamma''_{j, l})_{1\leq l\leq g_j}$ in the interior of $\Sigma_j$ such that
\begin{itemize}
\item the $\Gamma'_{j,l}$, $ 1\leq l\leq g_j$, are pairwise disjoint;
\item the $\Gamma''_{j,l}$, $ 1\leq l\leq g_j$, are pairwise disjoint;
\item $\Gamma'_{j,l'}\cap \Gamma''_{j,l''}=\emptyset$, if $l'\not=l''$;
\item $\Gamma'_{j,l}\cap \Gamma''_{j,l}$ is reduced to point and we have $[\Gamma'_{j,l}]\wedge[\Gamma''_{j,l}]=1$.
\end{itemize}
Noting that $H\subset H'$, we deduce that the $\Gamma_k$, $1\leq k\leq h$, and the $\Gamma'_{j,l}$, $ 1\leq l\leq g_j$, $j\in J$, generate an isotropic subspace of $H_1(S,\R)$ of dimension $g$, meaning a Lagrangian subspace. \end{proof}

\begin{lemma}  \label{l:homoclinic} Consider $f\in \cG^r_{\omega}(S)$ isotopic to the identity and an identity isotopy $I$ of $f$. We suppose that $f$ has homoclinic classes. Suppose that $\kappa\in\mathcal E(f)$ contains a point $z\in \mathrm{fix}_{I,\mathrm{h^+}}(f)$ and that $\iota_*(H_1(\kappa,\R))\not=\{0\}$.  If $\tilde z\in \tilde S$ is a lift of $z$ and $\Lambda_s$, $\Lambda_u$ are branches of $\tilde z$, respectively stable and unstable, there exists $T\in G$ with $[T]\not=0$, such that  $\Lambda_u$ intersects transversally $T(\Lambda_s)$.
\end{lemma}

\begin{figure}[h]
\centering
  \def\svgwidth{\columnwidth}
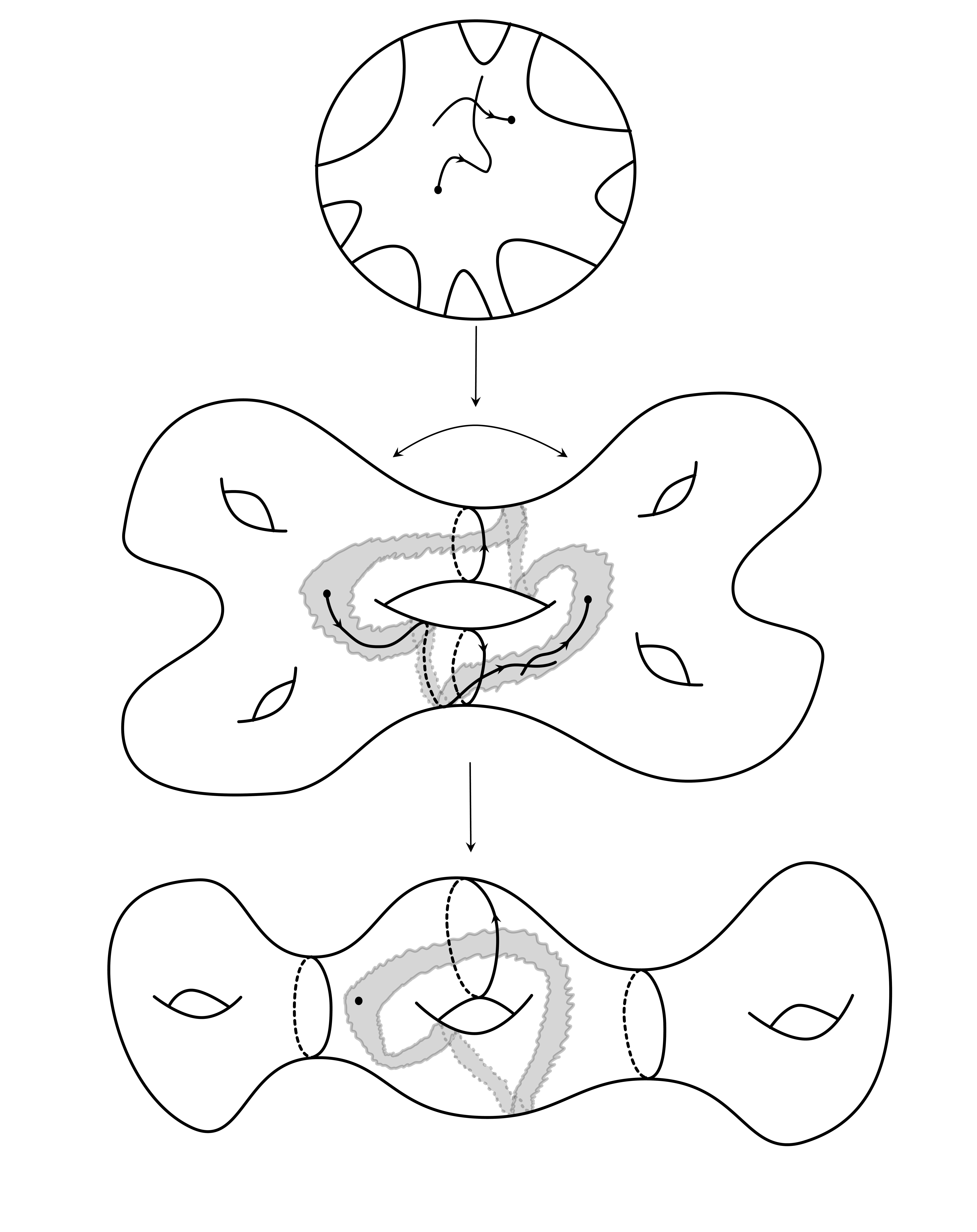
\caption{}
\label{f.homoclinic}
\end{figure}

\begin{proof}
Let $V$ be a minimal regular neighborhood of $\kappa$. It is invariant by $f$ because $f(\kappa)=\kappa$. By assumption, $\iota_*(H_1(V,\R))\not=\{0\}$ and so there exists $\alpha\in H_1(V,\R)$ such that $i_*(\alpha)\not=0$. We can find  a base of $H_1(S,\R)$ that consists of homological classes of simple loops. It implies that there exists a simple loop $\Gamma\subset S$ such that $i_*(\alpha)\wedge[\Gamma]\not=0$.
We denote by $S'$ the $2$-sheet covering space of $S$ obtained by cutting and gluing two copies of $S$ along $\Gamma$. We denote $\pi': S'\to S$ the covering projection and $\tau: S'\to S'$ the non trivial covering automorphism. There is a lift $f'$ of $f$ to $S'$ such that $I$ can be lifted to an identity isotopy $I'$ of $f'$. Moreover it holds that $f'\in \cG^r_{\pi'{}^*(\omega)}(S)$. The lifts $z'_1$ and $z'_2$ of $z$ belong to $ \mathrm{fix}_{I',\mathrm{h^+}}(f')$ and we have $\tau(z'_1)=z'_2$, $\tau(z'_2)=z'_1$. See Figure \ref{f.homoclinic}. We will prove by contradiction that $z'_1$ and $z'_2$ are equivalent, meaning that the class $\kappa'_1\in{\mathcal E}(f')$ of $z'_1$ and the class $\kappa'_2\in{\mathcal E}(f')$ of $z'_2$ coincide
. Of course it holds that $\tau(\kappa'_1)=\kappa'_2$ and $\tau(\kappa'_2)=\kappa'_1$. Suppose that $\kappa'_1\not=\kappa'_2$. Then, by Proposition \ref{p.separationconsequences} and Proposition \ref{p.minimal}, one can find $V'_1$, $V'_2$ regular neighborhoods of $\kappa'_1$, $\kappa'_2$ respectively, such that $V'_1\cap V'_2=\emptyset$.  Denote $W'_1$ the connected component of $V'_1\cap \tau(V'_2)$ that contains $\kappa'_1$ and $W'_2$ the connected component of $V'_2\cap \tau(V'_1)$ that contains $\kappa_2$. The domains $W'_1$ and $W'_2$ are disjoint regular neighborhoods of $\kappa'_1$, $\kappa'_2$ respectively and it holds that $\tau(W'_1)=W'_2$, $\tau(W'_2)=W'_1$.  By definition of a regular domain, both sets $W'_1$ and $W'_2$ are invariant by $f'$. So, they project onto the same open set $W\subset S$, which contains $\kappa$ and is invariant by $f$. More precisely the map $\pi'$ induces a homeomorphism from $W'_1$ to $W$ and a homeomorphism from $W'_2$ to $W$. In fact, the domains $W'_1$ and $W'_2$ are the connected component of $\pi'{}^{-1}(W)$. The set $W$ has finite type because it is the case of $W'_1$ and $W'_2$. Moreover, its complement in $S$ has no isolated point because the complements of $W'_1$ and $W'_2$ in $S'$ have no isolated point. Consequently, $W$ is a regular neighborhood of $\kappa$. So, there exists $\alpha'\in H_1(W,\R)$ such that $i_*(\alpha')=i_*(\alpha)$. We have $i_*(\alpha)\wedge[\Gamma]\not=0$. It implies that there is a loop $\Gamma'\in W$ such that $[\Gamma']\wedge [\Gamma]\not =0$. We deduce that $\pi'^{-1}( \Gamma')$ is a loop. This contradicts the fact that $\pi'^{-1}(W)$ is not connected.

The map $f'$ has homoclinic classes because it is the case for $f$.  Switching $z'_1$ and $z'_2$ if necessary, one can suppose that $\tilde z$ is a lift of $z'_1$. The unstable branch $\Lambda'_{u,1}$ of $z'_1$ that lifts $\Lambda_u$ and the stable branch  $\Lambda'_{s,2}$ of $z'_2$ that lifts $\Lambda_s$ intersect transversally at a point $z'_3$. Consider the path $\lambda'=\lambda'_1\lambda'_2$, where $\lambda'_1$ is the segment of $\Lambda'_{u,1}$ that joins $z'_1$ to $z'_3$ and $\lambda'_2$ is the segment of $\Lambda'_{s,2}$ that joins $z'_3$ to $z'_2$. One can lift $\lambda'_1$ to a path $\tilde \lambda_1$ joining $\widetilde z$ to a lift $\tilde z_3$ of $z'_3$. Moreover, there exists $T\in G$ such that $\lambda'_2$ can be lifted to a path $\tilde \lambda_2$ joining $\widetilde z_3$ to $T(\tilde z)$. The unstable branch of $\tilde z_1$ that lifts $\Lambda_u$  intersects transversally  at $\tilde z_3$ the stable  branch of $T(\tilde z)$ that lifts $\Lambda_s$. Observe that, by Proposition \ref{p.regulardomain-class}, $T$ is in the stabilizer of $\pi^{-1}(V).$  A path  joining $\tilde z$ to $T(\tilde z)$ projects in $S'$ to a path joining $z'_1$ to $z'_2$. So, $[T]\wedge [\Gamma]$ is an odd integer, which implies that $[T]$ does not vanish. \end{proof}

\textbf{Proof of Proposition \ref{p.keyresult}:} Based on the previous lemmas the proof of Theorem  is now very easy. Suppose that $f\in \cG^r_{\omega}(S)$ is isotopic to the identity and that 

\begin{itemize}
\item either $g\geq 2$ and $f$ has more than $2g-2$ periodic points;

\item or $g=1$ and $\mathrm{fix}_{I}(f)\not=\emptyset$.

\end{itemize}
In both cases $\mathrm{fix}_{I,\mathrm{h^+}}(f)$ is not empty. One can define the family $(\kappa_j)_{j\in J}$ of classes that contain at least one 
point of $\mathrm{fix}_{I,\mathrm{h^+}}(f)$. By Proposition \ref{p.separationconsequences} and Proposition \ref{p.minimal}, one can find a family of disjoints subsets $(V_j)_{j\in I} $ of $S$, where $V_j$ is a minimal regular neighborhood of $\kappa_j$. Of course every $V_j$ is fixed and every connected component of $\tilde \pi^{-1}(V_j)$ is fixed by $\tilde f_I$.  One can apply Lemma \ref{l:nonzerohomology}: there exists $j\in J$ such that $\iota_*(H_1(V_j,\R))\not=\{0\}$. Moreover, by Theorem \ref{t.numberhyp}, $f$ has homoclinic classes, and so, applying  Lemma \ref{l:homoclinic}, we deduce that for every $z\in \mathrm{fix}_{I,\mathrm{h^+}}(f)\cap \kappa_j$ for every lift $\tilde z\in \tilde S$ of $z$, for every branches $\tilde\Lambda_s$, $\tilde \Lambda_u$ of $\tilde z$, respectively stable and unstable, there exists $T\in G$ satisfying $[T]\not=0$, such $\tilde \Lambda_u$ and $T( \tilde \Lambda_s)$ intersect transversally.
\vskip 5pt

\textbf{Proof of Proposition \ref{p.precisekeyresult}:} The proof of this proposition follows the same lines as above. Indeed, it will be an inductive argument, the first step being Proposition \ref{p.keyresult} (and its proof). As before, one can find a homoclinic class $\ka$ associated to a fixed point $z\in \mathrm{fix}_{I,\mathrm{h^+}}(f)$ and a minimal regular domain $V$ of $\ka$ such that $\iota_*(H_1(V,\R))\not=\{0\}.$ We have that $V$ is fixed and every connected component of $\pi^{-1}(V)$ is fixed by $\tilde f_I.$

Assume that there exist $T_1,..., T_q\in G$  in the stabilizer of $\pi^{-1}(V)$ such that, for a lift $\tilde z$ of $z$, and for branches $\tilde\Lambda_s$, $\tilde\Lambda_u$ of $\tilde z$, respectively stable and unstable, we have
\begin{itemize}
\item $\tilde\Lambda_u$  intersects transversally $T_i(\tilde\Lambda_s)$ for $1\le i\le q;$
\item the classes $[T_i]\subset \iota_*(H_1(V,\R))$ are linearly independent.
\end{itemize}

Proposition \ref{p.keyresult} says that the above is true for $q=1.$ Assume that $\{[T_i];1\le i\le q\}$ do not generate $\iota_*(H_1(V,\R)).$ We would like to find $T_{q+1}$ in the stabilizer of $\pi^{-1}(V)$ such that $\tilde \Lambda_u$  intersects transversally $T_{q+1}(\tilde\Lambda_s)$ and that $\{[T_i];1\le i\le q+1\}$ is linearly independent.  If so, the proof of Proposition \ref{p.precisekeyresult} is done.

Assume that $\{[T_i];1\le i\le q\}$ do not generate $\iota_*(H_1(V,\R))$.  Let $H$ be the subspace of $H_1(S,\R)$ generated by $\{[T_i];1\le i\le q\}$. We have that $(H^\perp)^\perp=H$ has dimension $q$ (where $\perp$ is with respect to the wedge product $\wedge$). Then, there exists   $\alpha\in H_1(V,\R)$ such that $i_*(\alpha)\not\in (H^\perp)^\perp. $ In particular, there exists $\beta\in H^\perp$ such that $\alpha \wedge \beta\neq 0.$ Since $[T_i]$ belongs to $H_1(S, \Z)$ we may consider that $\beta $ is in $H_1(S,\Z)$ as well. Moreover, we may assume that $\beta$ is primitive\footnote{An element of $h\in H_1(S,\Z)$ is called \textit{primitive} if it can not be written as $h=nh$ with $h'\in H_1(S,\Z)$ and $n\ge 2.$}. Primitive elements can be represented by simple loops  (see \cite{Pu} or \cite{Me}). Therefore, there is a simple loop $\Gamma$ such that $\beta=[\Gamma].$  Observe that $i_*(\alpha)\wedge [\Gamma]\neq 0$ and $ [T_i]\wedge [\Gamma]=0$ if $1\le i\le q.$

Now, the arguments in Lemma \ref{l:homoclinic} lead us to the existence of $T_{q+1}$ in the stabilizer of $\pi^{-1}(V)$ such that 
\begin{itemize}
\item $\tilde \Lambda_u$  intersects transversally $T_{q+1}(\tilde\Lambda_s)$
\item $[T_{q+1}]\in \iota_*(H_1(V,\R))$
\item $[T_{q+1}]\wedge [\Gamma]\neq 0.$
\end{itemize}
The last item implies that $\{[T_i]:1\le i\le q+1\}$ is linearly independent. Thus, by  induction, we complete the proof of Proposition \ref{p.precisekeyresult}.

\section{A criteria of existence of periodic orbits with non trivial rotation vector}\label{s.poincarebirkhofflike}

In this section, we give an alternate way to prove the existence of non contractible periodic orbits in certain special situations. We set it in the case where $g\geq 2$ but one can do similar things in the case of a torus   We will prove the following.

\begin{proposition} Consider $f\in \cG^r_{\omega}(S)$, isotopic to the identity, where $g\geq 2$. Let $V$ be an invariant regular domain such that the connected components of $\tilde\pi^{-1}(V)$ are fixed by the canonical lift $\tilde f$. Let $\Sigma$ be a finite approximation of $V$ and $\Gamma$ a boundary circle of $\Sigma$. We suppose that $\Gamma$ is not null homotopic. Then there exist a real number $a>0$ such that for every $p/q\in(0,a)$, written in an irreducible way, there exists a non contractible periodic orbit $O_{p/q}$ of period $q$, these orbits being all distinct. Moreover, there exists an orientation of $\Gamma$ such that for every $p/q\in(0,a)$, we have $\mathrm{rot}_f(O_{p/q})=(p/q)[\Gamma]$, where $[\Gamma]\in H_1(S,\Z)$ is the homology class of the oriented loop $\Gamma$.

\end{proposition}

\begin{proof}  Let $\tilde \Gamma$ be a lift of $\Gamma$ to $\tilde S$ and $T$  one of the two generators of the stabilizer of $\Gamma$ in $G$. The canonical lift $\tilde f$ commutes with $T$ and lifts a diffeomorphism $\hat f$ of the annular covering space $\hat S=\tilde S/T$ of $S$. One can identify $\tilde S$ with the unit disk $\D$ and $G$ with a group of hyperbolic M\"obius automorphisms of $\D$.  It is well known that $\tilde f$ extends naturally to a homeomorphism of $\overline \D$ that fixes every point of $\S^1$. The topological line $\tilde\Gamma$ joins a point $\alpha(\tilde \Gamma)\in \S^1$
to a point  $\omega(\tilde \Gamma)\in \S^1$. Denote $J_r$ the connected component of $\S^1\setminus (\alpha(\tilde \Gamma)\cup \omega(\tilde \Gamma))$ that joins $\alpha(\tilde \Gamma)$ to $\omega(\tilde \Gamma)$ for the usual orientation and $J_l$ the connected component of $\S^1\setminus (\alpha(\tilde \Gamma)\cup \omega(\tilde \Gamma))$ that joins $\omega(\tilde \Gamma)$ to $\alpha (\tilde \Gamma)$. 
The map $\hat f$ extends to the compact annulus $\left(\overline D\setminus (\alpha(\tilde \Gamma)\cup \omega(\tilde \Gamma)\right)/T$ that fixes every point of the two boundary loops $J_r/T$ and $J_l/T$.   

Denote $U$ the connected component of $V\setminus\Sigma$ that is bordered by $\Gamma$ and recall that it is an annular punctured neighborhood of the end $e(\Gamma)$. 
Denote $\tilde U$ the connected component of $\tilde\pi^{-1}(U)$ that is bordered by $\tilde \Gamma$. Finally set $\hat U=\tilde U/T$. Denote $\hat \pi: \hat S\to S$ the covering projection and  $\hat \omega =\hat \pi^*(\omega)$.  Note that $\hat f$ preserves $\hat\omega$, that  $\hat f_{\vert \hat U}$ is a homeomorphism from $\hat U$ to $U$ and that $\mu_{\hat\omega}(\hat U)=\mu_{\omega}(U)$. There is no loss of generality by supposing that $\hat U$ is on the right of $\hat\Gamma$. One gets an invariant annulus $\hat A$ by taking the union of $\hat U$ and of the closure of the connected component of $\hat S\setminus\hat \Gamma$  located on the left of $\hat\Gamma$. Its universal covering space $\tilde A$ is the union of $\tilde U$ and of  the closure of the connected component of $\D\setminus\tilde\Gamma$  located on the left of $\tilde \Gamma$. One of the ends of $\hat A$ can be blown up by adding $J_l/T$. The other end can be blown up by adding the circle of prime ends at $e(\Gamma)$ via the map $\hat \pi$. The homeomorphism $\hat f_{\vert \hat A}$ can be extended by the identity on  $J_l/T$ and by a homeomorphism of irrational number on the circle of prime ends at $e(\Gamma)$. The universal covering space of $\hat A$ is a subset of $\D$ and one can add to $\hat A$ the line $J_l$ and the universal covering space of the circle of prime ends at $e(\Gamma)$. There is a lift that fixes every point of $J_l$ and coincides with $\tilde f$. The measure $\mu_{\hat \omega}$ is infinite, but is locally finite in a neighborhood of one of the ends of $\hat A$ because  $\mu_{\hat\omega}(\hat U)<+\infty$. Consequently, $\hat f$ satisfies the following intersection property: every simple loop of $\hat A$ non homotopic to zero meets its image by $\hat f$.  Replacing $T$ by its inverse if necessary, one can suppose that the rotation  number $a$ induced on the the universal covering space of the circle of prime ends at $e(\Gamma)$ is a positive irrational number. So, one can apply the classical Poincar\'e-Birkhoff theorem: for every $p/q\in(0,a)$, there exists $\tilde z_{p/q}\in \tilde D$ such that $\tilde f^q(z_{p/q})=T^{p}(\tilde z)$. It is easy to prove that the orbit of $\tilde z_{p/q}$ projects onto a periodic orbit $O_{p/q}$ of $f$ of period $q$ (because $p$ and $q$ are relatively prime) and that the orbits $O_{p/q}$ are all distinct.

\end{proof}

\end{document}